\documentclass[12pt]{amsart}
\usepackage{amssymb,amsmath,amscd}
\usepackage{latexsym,bm,bbm,mathrsfs}
\usepackage{hyperref,graphicx}
\usepackage[all,pdf]{xy}

\newcommand{\be}{\begin{equation}}
\newcommand{\ee}{\end{equation}}

\newtheorem{thm}{Theorem}[section]

\newtheorem{cor}[thm]{Corollary}
\newtheorem{prop}[thm]{Proposition}
\newtheorem{lemma}[thm]{Lemma}

\newtheorem{remark}[thm]{Remark}

\newcommand{\bba}{{\mathbb{A}}}

\newcommand{\bbq}{{\mathbb{Q}}}

\newcommand{\bbz}{{\mathbb{Z}}}
\newcommand{\bbc}{\mathbb{C}}
\newcommand{\bbf}{\mathbb{F}}

\newcommand{\Gal}{\operatorname{Gal}}

\newcommand{\Tr}{\operatorname{Tr}}

\newcommand{\ord}{\operatorname{ord}}

\newcommand{\Hom}{\operatorname{Hom}}

\title[Hilbert schemes of points and generalized Kummer varieties]{Hilbert schemes of points on smooth projective surfaces and generalized Kummer varieties with finite group actions}

\author{Sailun Zhan}
\address{Department of Mathematical Sciences, Binghamton University, Binghamton, NY, 13902, U.S.A.}
\email{zhans@binghamton.edu}
\begin{document}

\begin{abstract}
G\"ottsche and Soergel gave formulas for the Hodge numbers of Hilbert schemes of points on a smooth algebraic surface and the Hodge numbers of generalized Kummer varieties. When a smooth projective surface $S$ admits an action by a finite group $G$, we describe the action of $G$ on the Hodge pieces via point counting. Each element of $G$ gives a trace on $\sum_{n=0}^{\infty}\sum_{i=0}^{\infty}(-1)^{i}H^{i}(S^{[n]},\mathbb{C})q^{n}$. In the case that $S$ is a K3 surface or an abelian surface, the resulting generating functions give some interesting modular forms when $G$ acts faithfully and symplectically on $S$. 
\end{abstract}
\maketitle

\section{Introduction}

Let $S$ be a smooth projective surface over $\bbc$. In \cite{GS93}, the Hodge numbers of the Hilbert scheme of points of $S$ are computed via perverse sheaves/mixed Hodge modules:
\[
\sum_{n=0}^{\infty} h(S^{[n]},u,v)t^n=\prod_{m=1}^{\infty} \prod_{p,q}\left({\sum_{i=0}^{h^{pq}}(-1)^{i(p+q+1)}\binom{h^{pq}}{i}u^{i(p+m-1)}v^{i(q+m-1)}t^{mi}} \right)^{(-1)^{p+q+1}},
\]
where $S^{[n]}$ is the Hilbert scheme of $n$ points of $S$, $h(S^{[n]},u,v)=\sum_{p,q}h^{pq}(S^{[n]})u^{p}v^{q}$ is the Hodge-Deligne polynomial, and $h^{pq}$ are the dimensions of the Hodge pieces $H^{p,q}(S,\bbc)$. The Hodge numbers of the higher order Kummer varieties (generalized Kummer varieties) of an abelian surface are also computed:
\[
h(K_n(A),-u,-v)=
\]
\[
\frac{1}{((1-u)(1-v))^2}\sum_{\alpha\in P(n)}gcd(\alpha)^{4}(uv)^{n-|\alpha|}\left(\prod_{i=1}^{\infty}\sum_{\beta^{i}\in P(\alpha_i)}\prod_{j=1}^{\infty}\frac{1}{j^{\beta^{i}_j}\beta^{i}_{j}!}((1-u^j)(1-v^j))^{2\beta^{i}_j}\right),
\]
where $\alpha=(1^{\alpha_1}2^{\alpha_2}...)$ is a partition of $n$, $|\alpha|$ is the number of parts, and $gcd(\alpha):=gcd\{i\in\bbz|\alpha_i\neq 0\}$.

In this paper $G$ will always be a finite group. We will consider a smooth projective K3 surface $S$ over $\mathbb{C}$ with a $G$-action, and ask whether we can prove similar equalities for $G$-representations. We use an equivariant version of the idea in G\"ottsche \cite{Go90}, which studies the cohomology groups by counting the number of rational points over finite fields. Then we lift the results to the Hodge level by p-adic Hodge theory.

We will consider the G-equivariant Hodge-Deligne polynomial for a smooth projective variety $X$
\[
E(X;u,v)=\sum_{p,q}(-1)^{p+q}[H^{p,q}(X,\bbc)]u^{p}v^{q},
\]
where the coefficients lie in the ring of virtual G-representations $R_{\bbc}(G)$, of which the elements are the formal differences of isomorphism classes of finite dimensional $\bbc$-representations of $G$. The addition is given by direct sum and the multiplication is given by tensor product.

\begin{thm}\label{thm1}
Let $S$ be a smooth projective surface over $\bbc$ with a $G$-action. Let $S^{[n]}$ be the Hilbert scheme of n points of $S$. Then we have the following equality as virtual $G$-representations.
\[
\sum_{n=0}^{\infty}E(S^{[n]})t^{n}=\prod_{m=1}^{\infty}\prod_{p,q}\left({\sum_{i=0}^{h_{p,q}}(-1)^{i}[\wedge^{i}H^{p,q}(S,\bbc)]u^{i(p+m-1)}v^{i(q+m-1)}t^{mi}} \right)^{(-1)^{p+q+1}},
\]
where $h_{p,q}$ are the dimensions of the Hodge pieces $H^{p,q}(S,\bbc)$.
\end{thm}

\begin{remark}
Theorem \ref{thm1} has been proved in \cite[Theorem 1.1]{Zha21}, where the proof uses Nakajima operators. We give a new proof here using the Weil conjecture and p-adic Hodge theory.
\end{remark}

For a complex K3/abelian surface $S$ with an automorphism $g$ of finite order $n$, $H^{0}(S,K_{S})=\mathbb{C}\omega_{S}$ has dimension 1, and we say $g$ acts symplectically on $S$ if it acts trivially on $\omega_{S}$, and $g$ acts non-symplectically otherwise, namely, $g$ sends $\omega_{S}$ to $\zeta_{n}^{k}\omega_{S}$, $0<k<n$, where $\zeta_{n}$ is a primitive $n$-th root of unity.

Denote by $[e(X)]$ the virtual graded G-representation $\sum_{i=0}^{\infty}(-1)^{i}[H^{i}(X,\bbc)]$ for a smooth projective variety $X$ over $\bbc$ with a G-actoin.

\begin{thm}\label{thm2}
Let $G$ be a finite group which acts faithfully and symplectically on a smooth projective K3 surface $S$ over $\overline{\bbf}_{q}$. Suppose $p\nmid|G|$. Then
\[
\sum_{n=0}^{\infty}\Tr(g,[e(S^{[n]})])t^{n}=\exp\left(\sum_{m=1}^{\infty}\sum_{k=1}^{\infty}\frac{\epsilon(\ord(g^{k}))t^{mk}}{k}\right)
\]
for all $g\in G$, where $\epsilon(n)=24\left(n\prod_{p|n}\left(1+\frac{1}{p}\right)\right)^{-1}$. In particular, if $G$ is generated by a single element $g$ of order $N$, then we deduce that
\begin{center}
\begin{tabular}{c||c}
$N$ & $\sum_{n=0}^{\infty}\Tr(g,[e(S^{[n]})])t^{n}$ \\
\hline\hline
$1$ & $t/\eta^{24}(t)$ \\
$2$ & $t/\eta^{8}(t)\eta^{8}(t^{2})$ \\
$3$ & $t/\eta^{6}(t)\eta^{6}(t^{3})$ \\
$4$ & $t/\eta^{4}(t)\eta^{2}(t^{2})\eta^{4}(t^{4})$ \\
$5$ & $t/\eta^{4}(t)\eta^{4}(t^{5})$ \\
$6$ & $t/\eta^{2}(t)\eta^{2}(t^{2})\eta^{2}(t^{3})\eta^{2}(t^{6})$ \\
$7$ & $t/\eta^{3}(t)\eta^{3}(t^{7})$ \\
$8$ & $t/\eta^{2}(t)\eta(t^{2})\eta(t^{4})\eta^{2}(t^{8})$ 
\end{tabular}
\end{center}
where $\eta(t)=t^{1/24}\prod_{n=1}^{\infty}(1-t^{n})$.
\end{thm}

\begin{remark}
If $g$ acts symplectically on $S$, then $g$ has order $N\leq 8$ by \cite[Theorem 3.3]{DK09} since the $G$-action is tame. These eta quotients coincide with the results in the characteristic zero case. See \cite{BG19}, \cite[Lemma 3.1]{BO18}, or \cite{Zha21}.
\end{remark}

\begin{thm}\label{thm3}
Let $g$ be a symplectic automorphism (fixing the origin) of order $N$ on an abelian surface $S$ over $\bbc$. Then
\begin{center}
\begin{tabular}{c||c}
$N$ & $\sum_{n=0}^{\infty}\Tr(g,[e(S^{[n]})])t^{n}$ \\
\hline\hline
$1$ & $1$ \\
$2$ & $\eta^{8}(t^2)/{\eta^{16}(t)}$ \\
$3$ & $\eta^{3}(t^3)/{\eta^{9}(t)}$ \\
$4$ & ${\eta^{4}(t^4)}/{\eta^{4}(t)\eta^{6}(t^2)}$ \\
$6$ & ${\eta^{4}(t^{6})}/{\eta(t)\eta^{4}(t^2)\eta^{5}(t^3)}$ \\

\end{tabular}
\end{center}
\end{thm}

\begin{remark}
If $g$ is a symplectic automorphism on a complex abelian surface, then $g$ has order $1,2,3,4$ or $6$ by \cite[Lemma 3.3]{Fuj88}. These eta quotients coincide with the results of \cite[Theorem 1.1]{Pie21} when $G$ is cyclic.
\end{remark}

Define a multiplication $\odot$ on the ring of power series $R_{\bbc}(G)[[u,v,w]]$ by $u^{n_1}v^{m_1}w^{l_1}\odot u^{n_2}v^{m_2}w^{l_2}:=u^{n_1+n_2}v^{m_1+m_2}w^{gcd(l_1,l_2)}$. 

\begin{thm}\label{thm4}
Let $A$ be an abelian surface over $\bbc$ with a $G$-action. Let $K_n(A)$ be the generalized Kummer variety. Then we have the following equality as virtual $G$-representations.
\[
\sum_{n=0}^{\infty}E(K_n(A);u,v)t^n=\frac{(w\frac{d}{dw})^4}{E(A)}
\]
\[
\bigodot_{m=1}^{\infty}\left(1+w^{m}\left(-1+\prod_{p,q}\left({\sum_{i=0}^{h_{p,q}}(-1)^{i}[\wedge^{i}H^{p,q}(S,\bbc)]u^{i(p+m-1)}v^{i(q+m-1)}t^{mi}} \right)^{(-1)^{p+q+1}}\right)\right).
\]

\end{thm}

When we say $S$ is a surface with a $G$-action over a field $K$, we mean that both $S$ and the $G$-action can be defined over $K$. 

\begin{center}
\sc{Acknowledgements}
\end{center}
\vspace{0.1 in}
I thank Michael Larsen for many valuable discussions throughout this work.

\section{Preliminaries}
Let $X$ be a smooth projective variety over $\mathbb{C}$. Then we can choose a finitely generated $\mathbb{Z}$-subalgebra $\mathcal{R}\subset\mathbb{C}$ such that $X\cong\mathcal X\times_{\mathcal S}\text{Spec}\mathbb{C}$ for a regular projective scheme $\mathcal X$ over $\mathcal{S}=\text{Spec}\mathcal{R}$, and we can choose a maximal ideal $\mathfrak q$ of $\mathcal R$ such that $\mathcal X$ has good reduction modulo $\mathfrak q$. Since there are comparison theorems between \'etale cohomology and singular cohomology, we focus on characteristic $p$.

Now let $X$ be a quasi-projective variety over $\overline{\bbf}_{p}$ with an automorphism $\sigma$ of finite order. Suppose $X$ and $\sigma$ can be defined over some finite field $\mathbb{F}_q$. Let $F_q$ be the corresponding geometric Frobenius. Then for $n \geq 1$, the composite $F_{q}^{n}\circ\sigma$ is the Frobenius map relative to some new way of lowering the field of definition of $X$ from $\overline{\bbf}_{p}$ to $\mathbb{F}_{q^n}$ (\cite[Prop.3.3]{DL76} and \cite[Appendix(h)]{Ca85}). Then the Grothendieck trace formula implies that $\sum_{k=0}^{\infty}(-1)^{k}\text{Tr}((F_{q}^{n}\sigma)^{*},H_{c}^{k}(X,\mathbb{Q}_{l}))$ is the number of fixed points of $F_{q}^{n}\sigma$, where $H_{c}^{k}(X,\mathbb{Q}_{l})$ are the compactly supported $l$-adic cohomology groups.

\begin{lemma}\label{Weil}
Let $X$ and $Y$ be two smooth projective varieties over $\overline{\bbf}_{p}$ with finite group $G$-actions. Suppose $X,Y$ and the actions of G can be defined over $\bbf_{q}$, where $q$ is a $p$ power. If $|{X}(\overline{\mathbb{F}}_{p})^{gF_{q^n}}|=|{Y}(\overline{\mathbb{F}}_{p})^{gF_{q^n}}|$ for every $n\geq 1$ and $g\in G$, then $H^{i}(X,\bbq_{l})\cong H^{i}(Y,\bbq_{l})$ as $G$-representations for every $i\geq 0$.
\end{lemma}

\begin{proof}
Fix $g\in G$. Denote by $F_{q}$ the geometric Frobenius over $\bbf_{q}$. Since the finite group action is defined over $\bbf_{q}$, the action $g$ commutes with $F_{q}$ and the action of $g$ on the cohomology group is semisimple. There exists a basis of the cohomology group such that the actions of $g$ and $F_{q}$ are in Jordan normal forms simultaneously. Let $\alpha_{i,j},j=1,2,...,a_i$ (resp.  $\beta_{i,j},j=1,2,...,b_i$) denote the eigenvalues of $F_q$ acting on $H^{i}({X}, \mathbb{Q}_l)$ (resp. $H^{i}({Y}, \mathbb{Q}_l)$) in such a basis, where $a_i$ (resp. $b_i$) is the $i$-th betti number. Let $c_{i,j},j=1,2,...,a_i$ (resp.  $d_{i,j},j=1,2,...,b_i$) denote the eigenvalues of $g$ acting on the same basis of $H^{i}(X, \mathbb{Q}_l)$ (resp. $H^{i}({Y}, \mathbb{Q}_l)$). Then the Grothendieck trace formula (\cite[Prop.3.3]{DL76} and \cite[Appendix(h)]{Ca85}) implies that
\[
|{X}(\overline{\mathbb{F}}_{p})^{gF_{q^n}}|=\sum_{i=0}^{\infty}(-1)^{i}\text{Tr}((gF_{q^{n}})^{*},H^{i}({X}.\mathbb{Q}_{l}))
\]

Since $|{X}(\overline{\mathbb{F}}_{p})^{gF_{q^n}}|=|{Y}(\overline{\mathbb{F}}_{p})^{gF_{q^n}}|$ for every $n\geq 1$, we have 
\[
\sum_{i=0}^{\infty}(-1)^{i}\sum_{j=1}^{a_i}c_{i,j}\alpha_{i,j}^{n}=\sum_{i=0}^{\infty}(-1)^{i}\sum_{j=1}^{b_i}d_{i,j}\beta_{i,j}^{n}
\]
for every $n\geq 1$. By linear independence of the characters $\chi_{\alpha}:\mathbb{Z}^{+}\rightarrow\mathbb{C}, n\mapsto \alpha^n$ and the fact that $\alpha_{i,j},\beta_{i,j},j=1,2,...$ all have absolute value $q^{i/2}$ by Weil's conjecture, we deduce that $a_{i}=b_{i}$ and $\sum_{j=1}^{a_i}c_{i,j}=\sum_{j=1}^{b_i}d_{i,j}$ for each $i$. But since $g$ is arbitrary, this implies that the $G$-representations $H^{i}({X},\mathbb{Q}_{l})$ and $H^{i}({Y},\mathbb{Q}_{l})$ are the same. 
\end{proof}

\begin{prop}\label{symmetric}
Let $X$ be a smooth projective variety with a G-action over $\bbf_{q}$. Denote the dimension of $X$ by $N$. Then 
\[
\sum_{k=0}^{\infty} \sum_{i=0}^{\infty}(-1)^{i}[H^{i}(X_{\overline{\bbf}_{p}}^{(k)},\bbq_{l})]z^{i}t^{k}=\prod_{j=0}^{2N}\left(\sum_{i=0}^{b_j}(-1)^{i}[\wedge^{i}H^{j}({X_{\overline{\bbf}_{p}}},\bbq_{l})]z^{ij}t^{i}\right)^{(-1)^{j+1}},
\]
where the coefficients lie in $R_{\bbq_{l}}(G)$. 
\end{prop}

\begin{proof}
By the Weil conjectures, we have 
\[
\text{exp}(\sum_{r=1}^{\infty}|{X}(\mathbb{F}_{q^r})|\frac{t^r}{r})=\sum_{k=0}^{\infty}|{X}^{(k)}(\mathbb{F}_{q})|t^{k}=\sum_{k=0}^{\infty}|{X}^{(k)}(\overline{\bbf}_{p})^{F_{q}}|t^{k}=\prod_{j=0}^{2N}\left(\prod_{i=1}^{b_j}(1-\alpha_{j,i}t)\right)^{(-1)^{j+1}},
\]
where $\alpha_{j,i}$ are the eigenvalues of $F_{q}$ on $H^{j}(X_{\overline{\bbf}_{p}},\bbq_{l})$.

By the discussion at the beginning of the section and the Grothendieck trace formula, we deduce that
\[
\sum_{k=0}^{\infty}\sum_{m=0}^{\infty}(-1)^{m}\sum_{i}h_{k,m,i}\beta_{k,m,i}^{n}t^{k}=\prod_{j=0}^{2N}\left(\prod_{i=1}^{b_j}(1-g_{j,i}\alpha_{j,i}^{n}t)\right)^{(-1)^{j+1}}
\]
where $h_{k,m,i}$ (resp. $\beta_{k,m,i}$) are the eigenvalues of $g$ (resp. $F_q$) on $H^{m}({X}_{\overline{\bbf}_{p}}^{(k)},\bbq_{l})$, and $g_{j,i}$ are the eigenvalues of $g$ on $H^{j}({X}_{\overline{\bbf}_{p}},\mathbb{Q}_{l})$. Hence we deduce that the trace of $g$ on the left hand side equals the trace of $g$ on the right hand side for each graded piece in the equality in Proposition \ref{symmetric} by the proof of Lemma \ref{Weil}.
\end{proof}

We obtain the information of Hodge pieces via $p$-adic Hodge theory by using an equivariant version of the method in \cite[\S 4]{It03}. 

\begin{prop}\cite[I. 2.3]{Ser68}\label{galoisrep}
Let $K$ be a number field, $m,m'\geq 1$ be integers, and $l$ be a prime number. Let
\[
\rho\colon\Gal(\bar{K}/K)\to GL(m,\bbq_l),\ \ \ \rho'\colon\Gal(\bar{K}/K)\to GL(m',\bbq_l)
\]
be continuous $l$-adic $\Gal(\bar{K}/K)$-representations such that $\rho$ and $\rho'$ are unramified outside a finite set $S$ of maximal ideals of $\mathcal{O}_K$. If
\[
\Tr(\rho(Frob_{\mathfrak{p}}))=\Tr(\rho'(Frob_{\mathfrak{p}}))\ \ \text{for all maximal ideals }\mathfrak{p}\notin S,
\]
then $\rho$ and $\rho'$ have the same semisimplifications as $\Gal(\bar{K}/K)$-representations. Here $Frob_{\mathfrak{p}}$ is the geometric Frobenius at $\mathfrak{p}$.
\end{prop}

Let $p$ be a prime number and $F$ be a finite extension of $\bbq_p$. Let $\bbc_p$ be a $p$-adic completion of an algebraic closure $\bar{F}$ of $F$. Define $\bbq_p(0)=\bbq_p$, $\bbq_p(1)=(\varprojlim\mu_{p^n})\otimes_{\bbz_p}\bbq_p$, and for $n\geq 1$, $\bbq_p(n)=\bbq_p(1)^{\otimes n}$, $\bbq_p(-n)=\Hom(\bbq_p(n),\bbq_p)$. Moreover, we define $\bbc_p(n)=\bbc_p\otimes_{\bbq_p}\bbq_P(n)$, on which $\Gal(\bar{F}/F)$ acts diagonally. It is known that $(\bbc_p)^{\Gal(\bar{F}/F)}=F$ and $(\bbc_p(n))^{\Gal(\bar{F}/F)}=0$ for $n\neq 0$.

Let $B_{HT}=\oplus_{n\in\bbz}\bbc_p(n)$ be a graded $\bbc_p$-module with an action of $\Gal(\bar{F}/F)$. For a finite dimensional $\Gal(\bar{F}/F)$-representation $V$ over $\bbq_p$, we define a finite dimensional graded $F$-module $D_{HT}(V)$ by $D_{HT}(V)=(V\otimes_{\bbq_p}B_{HT})^{\Gal(\bar{F}/F)}$. The graded module structure of $D_{HT}(V)$ is induced from that of $B_{HT}$. In general, it is known that
\[
\dim_{F} D_{HT}(V)\leq\dim_{\bbq_p}V.
\]
If the equality holds, $V$ is called a \emph{Hodge-Tate representation}.

\begin{thm}\cite{Fal88}\cite{Tsu99}\label{decomposition}(Hodge-Tate decomposition)
Let $X$ be a proper smooth variety over $F$ and $k$ be an integer. The $p$-adic \'etale cohomology $H_{\acute{e}t}^{k}(X_{\bar{F}},\bbq_p)$ of $X_{\bar{F}}=X\otimes_{F}\bar{F}$ is a finite dimensional $\Gal(\bar{F}/F)$-representation over $\bbq_p$. Then, $H_{\acute{e}t}^{k}(X_{\bar{F}},\bbq_p)$  is a Hodge-Tate representation, Moreover, there exists a canonical and functorial isomorphism
\[
\bigoplus_{i+j=k}H^{i}(X,\Omega_{X}^{j})\otimes_{F}\bbc_{p}(-j)\cong H_{\acute{e}t}^{k}(X_{\bar{F}},\bbq_p)\otimes_{\bbq_p}\bbc_p
\]
of $\Gal(\bar{F}/F)$-representations, where $\Gal(\bar{F}/F)$ acts on $H^{i}(X,\Omega_{X}^{j})$ trivially and the right hand side diagonally.
\end{thm}

Now for a finite dimensional $\Gal(\bar{F}/F)$-representation $V$ over $\bbq_p$, suppose it is also a $G$-representation such that the $G$-action commutes with the $\Gal(\bar{F}/F)$-action. In this case, we call it a $\Gal(\bar{F}/F)$-$G$-representation and we define a $G$-representation over $F$:
\[
[h^{n}(V)]:=(V\otimes_{\bbq_p}\bbc_{p}(n))^{\Gal(\bar{F}/F)}.
\]
\begin{lemma}\label{hodgetatelemma}
Let $W_2$ be a Hodge-Tate $\Gal(\bar{F}/F)$-$G$-representation and
\[
0\to W_1\to W_2\to W_3\to 0
\]
be an exact sequence of finite dimensional $\Gal(\bar{F}/F)$-$G$-representations over $\bbq_p$. Then $W_1$ and $W_3$ are Hodge-Tate representations and
\[
[h^{n}(W_2)]=[h^{n}(W_1)]\oplus[h^{n}(W_3)]=[h^{n}(W_{1}\oplus W_{3})]
\]
as $G$-representations for all $n$.
\end{lemma}
\begin{proof}
It follows from \cite[Lemma 4.4]{It03} that $W_1$ and $W_3$ are Hodge-Tate representations and we have the following short exact sequence of $G$-representations
\[
0\to D_{HT}(W_1)\to D_{HT}(W_2) \to D_{HT}(W_3) \to 0,
\]
which implies that 
\[
[h^{n}(W_2)]=[h^{n}(W_1)]\oplus[h^{n}(W_3)]=[h^{n}(W_{1}\oplus W_{3})].
\]
\end{proof}

\begin{cor}\label{hodgenumber}
Let $X$ be a proper smooth variety over $F$ with a $G$-action. Then 
\[
H^{i}(X,\Omega_{X}^{j})=[h^{j}(H^{i+j}(X_{\bar{F}},\bbq_p)^{ss})]\ \text{as }G\text{-representations for all}\ i,j,
\]
where $H^{i+j}(X_{\bar{F}},\bbq_p)^{ss}$ denotes the semisimplification of $H^{i+j}(X_{\bar{F}},\bbq_p)$ as a $\Gal(\bar{F}/F)$-representation.
\end{cor}
\begin{proof}
By theorem \ref{decomposition}, if we take the $\Gal(\bar{F}/F)$-invariant of $H^{i+j}(X_{\bar{F}},\bbq_p)\otimes_{\bbq_p}\bbc_{p}(j)$, we have
\[
H^{i}(X,\Omega_{X}^{j})=[h^{j}(H^{i+j}(X_{\bar{F}},\bbq_p))].
\]
On the other hand, since $H^{i+j}(X_{\bar{F}},\bbq_p)$ is a $\Gal(\bar{F}/F)$-$G$ Hodge-Tate representation,
\[
[h^{j}(H^{i+j}(X_{\bar{F}},\bbq_p))]=[h^{j}(H^{i+j}(X_{\bar{F}},\bbq_p)^{ss})]
\]
by lemma \ref{hodgetatelemma}. Hence we are done. 
\end{proof}

\begin{thm}\label{padichodge}
Let $\mathcal{X}$ and $\mathcal{Y}$ be $n$-dimensional smooth projective varieties over a number field $K$ with $G$-actions. Suppose for all but finitely many good reductions, we have
\[
|{X}(\overline{\mathbb{F}}_{p})^{gF_{q^n}}|=|{Y}(\overline{\mathbb{F}}_{p})^{gF_{q^n}}|\text{ for every }n\geq 1\text{ and }g\in G,
\]
where $X,Y$ are the good reductions over $\bbf_q$. Then 
\[
H^{p,q}(\mathcal{X}_{\bbc})\cong H^{p,q}(\mathcal{Y}_{\bbc}). 
\]
for all $p,q$ as $G$-representations.
\end{thm}
\begin{proof}
By the proof of Lemma \ref{Weil} and Proposition \ref{galoisrep}, we deduce that $H^{i}(\mathcal{X}_{\bar{K}},\bbq_l)$ and $H^{i}(\mathcal{Y}_{\bar{K}},\bbq_l)$ have the same semisimplifications as $\Gal(\bar{K}/K)$-$G$-representations.

Now take a maximal ideal $\mathfrak{q}$ of $\mathcal{O}_K$ dividing $l$. Let $F$ be the completion of $K$ at $\mathfrak{q}$. Fix an embedding $\bar{K}\hookrightarrow\bar{F}$. Then we have an inclusion $\Gal(\bar{F}/F)\subset\Gal(\bar{K}/K)$. Therefore, $H^{i}(\mathcal{X}_{\bar{F}},\bbq_l)$ and $H^{i}(\mathcal{Y}_{\bar{F}},\bbq_l)$ have the same semisimplifications as $\Gal(\bar{F}/F)$-$G$-representations. By Corollary \ref{hodgenumber}, we conclude that 
\[
H^{q}(\mathcal{X}_{\bbc},\Omega_{\mathcal{X}_{\bbc}}^{p})\cong H^{q}(\mathcal{Y}_{\bbc},\Omega_{\mathcal{Y}_{\bbc}}^{p})
\]
for all $p,q$ as $G$-representations.
\end{proof}

\section{Hilbert scheme of points}

We denote by $X^{[n]}$ the component of the Hilbert scheme of a projective scheme $X$ parametrizing subschemes of length $n$ of $X$. For properties of Hilbert scheme of points, see references \cite{Iar77}, \cite{Go94} and \cite{Nak99}. 

\begin{lemma}\label{hilbert scheme}
Let $S$ be a smooth projective surface with a G-action over $\bbf_{q}$. Suppose $g\in G$ and let $F_{q}$ be the geometric Frobenius. Then
\[
\sum_{n=0}^{\infty}|S^{[n]}(\overline{\bbf}_{q})^{gF_{q}}|t^{n}=\prod_{r=1}^{\infty}\left(\sum_{n=0}^{\infty}|{\rm Hilb}^{n}(\widehat{\mathcal{O}_{S_{\overline{\bbf}_{q}},x}})(\overline{\bbf}_{q})^{g^{r}F_{q}^{r}}|t^{nr}\right)^{|P_{r}(S,gF_{q})|},
\]
where ${\rm Hilb}^{n}(\widehat{\mathcal{O}_{S_{\overline{\bbf}_{q}},x}})$ is the punctual Hilbert scheme of $n$ points at some $g^{r}F_{q}^{r}$-fixed point $x\in S(\overline{\bbf}_{q})$, and $P_{r}(S,gF_{q})$ is the set of primitive 0-cycles of degree $r$ of $gF_{q}$ on $S$, whose elements are of the form $\sum_{i=0}^{r-1}g^{i}F_{q}^{i}(x)$ with $x\in S(\overline{\bbf}_{q})^{g^{r}F_{q}^{r}}\backslash(\cup_{j<r}S(\overline{\bbf}_{q})^{g^{j}F_{q}^{j}})$.
\end{lemma}

\begin{proof}
Let $Z\in S^{[n]}(\overline{\bbf}_{q})^{gF_{q}}$. Suppose $(n_{1},...,n_{r})$ is a partition of $n$ and $Z=(Z_{1},...,Z_{r})$ with $Z_{i}$ being the closed subscheme of $Z$ supported at a single point with length $n_{i}$. Then Supp$Z$ decomposes into $gF_{q}$ orbits. We can choose an ordering $\leq$ on $S(\overline{\bbf}_{q})$. In each orbit, we can find the smallest $x_{j}\in S(\overline{\bbf}_{q})$. Suppose $Z_{j}$ with length $l$ is supported on $x_{j}$ and $x_{j}$ has order $k$. Then the component of $Z$ which is supported on the orbit of $x_{j}$ is determined by $Z_{j}$, namely, it is $\cup_{i=0}^{k-1}g^{i}F_{q}^{i}(Z_{j})$ with length $kl$. Also notice that $Z_{j}$ is fixed by $g^{k}F_{q}^{k}$. Hence, to give an element of $S^{[n]}(\overline{\bbf}_{q})^{gF_{q}}$ is the same as choosing some $gF_{q}$ orbits and for each orbit choosing some element in ${\rm Hilb}^{n}(\widehat{\mathcal{O}_{S_{\overline{\bbf}_{q}},x}})(\overline{\bbf}_{q})^{g^{k}F_{q}^{k}}$ for some $g^{k}F_{q}^{k}$-fixed point $x$ in this orbit such that the final length altogether is $n$. Combining all of these into power series, we get the desired equality. 
\end{proof}

The idea we used above is explained in detail in \cite[lemma 2.7]{Go90}. We implicitly used the fact that $\pi:(S^{[n]}_{(n)})_{red}\rightarrow S$ is a locally trivial fiber bundle in the Zariski topology with fiber ${\rm Hilb}^{n}(\mathbb{F}_{q}[[s,t]])_{red}$ \cite[Lemma 2.1.4]{Go94}, where $S^{[n]}_{(n)}$ parametrizes closed subschemes of length $n$ that are supported on a single point. 

We need the following key lemma.

\begin{lemma}\label{fixpoint}
Let $S$ be a smooth projective surface with a G-action over $\bbf_{q}$. If $x\in S(\overline{\bbf}_{q})^{gF_{q}}$, where $g\in G$ and $F_{q}$ is the geometric Frobenius, then
\[
|{\rm Hilb}^{n}(\widehat{\mathcal{O}_{S_{\overline{\bbf}_{q}},x}})(\overline{\bbf}_{q})^{gF_{q}}|=|{\rm Hilb}^{n}(\mathbb{F}_{q}[[s,t]])(\overline{\bbf}_{q})^{F_{q}}|.
\]
\end{lemma}

We will prove this lemma later in this section. 

From Lemma \ref{fixpoint}, we observe that $|{\rm Hilb}^{n}(\widehat{\mathcal{O}_{S_{\overline{\bbf}_{q}},x}})(\overline{\bbf}_{q})^{gF_{q}}|$ is a number independent of the choice of the $gF_{q}$-fixed point $x$.

We denote ${\rm Hilb}^{n}(\mathbb{F}_{q}[[s,t]])$ by $V_{n}$. Combining Lemma \ref{hilbert scheme} and Lemma \ref{fixpoint}, we deduce that
\[
\sum_{n=0}^{\infty}|S^{[n]}(\overline{\bbf}_{q})^{gF_{q}}|t^{n}=\prod_{r=1}^{\infty}\left(\sum_{n=0}^{\infty}|V_{n}(\overline{\bbf}_{q})^{F_{q}^{r}}|t^{nr}\right)^{|P_{r}(S,gF_{q})|}.
\]

Recall the following structure theorem for the punctual Hilbert scheme of points. 

\begin{prop}\cite[Prop 4.2]{ES87}\label{cell}
Let $k$ be an algebraically closed field. Then ${\rm Hilb}^{n}(k[[s,t]])$ over $k$ has a cell decomposition, and the number of $d$-cells is $P(d,n-d)$, where $P(x,y):=\#$$\{\text{partition of x into parts}\leq y\}$.
\end{prop}

Denote by $p(n,d)$ the number of partitions of $n$ into $d$ parts. Then $p(n,d)=P(n-d,d)$. Now we can proceed similarly as in the proof of \cite[Lemma 2.9]{Go90}. 

\begin{proof}[Proof of Theorem \ref{thm1}]

Since we have 
\[
\prod_{i=1}^{\infty}\left(\frac{1}{1-z^{i-1}t^{i}}\right)=\sum_{n=0}^{\infty}\sum_{i=0}^{\infty}p(n,n-i)t^{n}z^{i},
\]
by Proposition \ref{cell} we get
\[
\sum_{n=0}^{\infty}\sum_{m=0}^{\infty}\#\{\text{m-dim cells of }{\rm Hilb}^{n}(\overline{\bbf}_{p}[[s,t]])\}t^{n}z^{m}=\prod_{i=1}^{\infty}\frac{1}{1-z^{i-1}t^{i}}.
\]
Fix $N\in\mathbb{N}$. Then by choosing sufficiently large $q$ powers $Q$ such that the cell decomposition of $V_{n,\overline{\bbf}_{q}}$ is defined over $\mathbb{F}_{Q}$ for $n\leq N$, we deduce that
\[
\sum_{n=0}^{\infty}|V_{n,\overline{\bbf}_{q}}(\mathbb{F}_{Q^{r}})|t^{nr}\equiv\prod_{i=1}^{\infty}\frac{1}{1-Q^{r(i-1)}t^{ri}}\ \ \ \text{mod }t^N.
\]
Now consider a good reduction of $S$ over $\bbf_q$.
\[
\begin{aligned}
\sum_{n=0}^{\infty}|S^{[n]}(\overline{\bbf}_{q})^{gF_{Q}}|t^{n}&\equiv\prod_{r=1}^{\infty}\prod_{i=1}^{\infty}\left(\frac{1}{1-Q^{r(i-1)}t^{ri}}\right)^{|P_{r}(S,gF_{Q})|}\ \ \ \text{mod }t^N\\
&=\text{exp}\left(\sum_{i=1}^{\infty}\sum_{r=1}^{\infty}\sum_{h=1}^{\infty}|P_{r}(S,gF_{Q})|Q^{hr(i-1)}t^{hri}/h\right)\\
&=\text{exp}\left(\sum_{i=1}^{\infty}\sum_{m=1}^{\infty}(\sum_{r|m}r|P_{r}(S,gF_{Q})|)Q^{m(i-1)}t^{mi}/m\right)\\
&=\prod_{i=1}^{\infty}\text{exp}\left(\sum_{m=1}^{\infty}|S(\overline{\bbf}_{q})^{g^{m}F_{Q}^{m}}|Q^{m(i-1)}t^{mi}/m\right)\\
&=\prod_{i=1}^{\infty}\sum_{n=0}^{\infty}|S^{(n)}(\overline{\bbf}_{q})^{gF_{Q}}|Q^{n(i-1)}t^{ni}.
\end{aligned}
\]
By replacing $Q$ by $Q$-powers and using the proof of Proposition \ref{symmetric} and Theorem \ref{padichodge}, we obtain
\[
\sum_{n=0}^{\infty}E(S^{[n]})t^{n}=\prod_{m=1}^{\infty}\prod_{p,q}\left({\sum_{i=0}^{h_{p,q}}(-1)^{i}[\wedge^{i}H^{p,q}(S,\bbc)]u^{i(p+m-1)}v^{i(q+m-1)}t^{mi}} \right)^{(-1)^{p+q+1}},
\]
since we can reduce to the case where everything is defined over a number field K as in \cite[Prop. 5.1]{It03}.
\end{proof}

\begin{cor} 
For a smooth projective surface $S$ over $\overline{\bbf}_{p}$ or $\bbc$, we have
\[
\sum_{n=0}^{\infty}[e(S^{[n]})]t^{n}=\prod_{m=1}^{\infty}\prod_{j=0}^{4}\left(\sum_{i=0}^{b_{j}}(-1)^{i}[\wedge^{i}H^{j}(S,\mathbb{Q}_{l})][-2i(m-1)]t^{mi}\right)^{(-1)^{j+1}},
\]
where the coefficients lie in $R_{\bbq_{l}}(G)$, and $[-2i(m-1)]$ indicates shift in degrees.
\end{cor}

\begin{remark}\label{rmk} 
Notice that the generating series of the topological Euler characteristic of $S^{[n]}$ is $\sum_{n=0}^{\infty}e(S^{[n]})t^{n}=\prod_{m=1}^{\infty}(1-t^{m})^{-e(S)}$. But this is not the case if we consider $G$-representations and regard $\prod_{m=1}^{\infty}(1-t^{m})^{-[e(S)]}$ as 
\[
\exp(\sum_{m=1}^{\infty}[e(S)](-\log(1-t^{m})))=\exp(\sum_{m=1}^{\infty}[e(S)](\sum_{k=1}^{\infty}t^{mk}/k))). 
\]
What we have is actually 
\[
\sum_{n=0}^{\infty}{\Tr}(g,[e(S^{[n]})])t^{n}=\prod_{m=1}^{\infty}\left(\frac{(\prod_{i=1}^{b_1}(1-{g_{1,i}}t^{m}))(\prod_{i=1}^{b_3}(1-{g_{3,i}}t^{m}))}{(1-t^{m})(\prod_{i=1}^{b_2}(1-{g_{2,i}}t^{m}))(1-t^{m})}\right)
\] 
\[
=\exp\left(\sum_{m=1}^{\infty}\sum_{k=1}^{\infty}\frac{t^{mk}}{k}\left(1-\sum_{i=1}^{b_{1}}g_{1,i}^{k}+\sum_{i=1}^{b_{2}}g_{2,i}^{k}-\sum_{i=1}^{b_{3}}g_{3,i}^{k}+1\right)\right).
\]
We will use this expression to determine the $G$-representation $[e(S^{[n]})]$ later when $S$ is a K3 surface or an abelian surface.
\end{remark}

Now we start to prove lemma \ref{fixpoint}.

Let $S$ be a smooth projective surface over $\mathbb{F}_{q}$ with an automorphism $g$ over $\mathbb{F}_{q}$ of finite order. If $x\in S(\overline{\bbf}_{q})^{gF_{q}}$ where $F_{q}$ is the geometric Frobenius, then $x$ lies over a closed point $y\in S$. Denote the residue degree of $y$ by $N$. Hence $x\in S(\mathbb{F}_{q^{N}})$ and there are $N$ geometric points $x, F_{q}(x),..., F_{q}^{N-1}(x)$ lying over $y$. 

Let us study the relative Hilbert scheme of $n$ points at a closed point.
\[
{\rm  Hilb}^{n}({\rm Spec}(\widehat{\mathcal{O}_{S,y}})/{\rm Spec}\mathbb{F}_{q})\cong{\rm Hilb}^{n}({\rm Spec}(\mathbb{F}_{q^{N}}[[s,t]])/{\rm Spec}\mathbb{F}_{q}).
\]

Since $g$ and $\mathbb{F}_{q}$ fix $y$, they act on this Hilbert scheme. Over $\overline{\bbf}_{q}$, we have
\[
 {\rm  Hilb}^{n}({\rm Spec}(\widehat{\mathcal{O}_{S,y}})/{\rm Spec}\mathbb{F}_{q})\otimes_{\mathbb{F}_{q}}\overline{\bbf}_{q}\cong{\rm Hilb}^{n}({\rm Spec}(\overline{\bbf}_{q}\otimes_{\mathbb{F}_{q}}\mathbb{F}_{q^{N}}[[s,t]])/{\rm Spec}\overline{\bbf}_{q})
\]
 by the base change property of the Hilbert scheme. Denote by $u$ a primitive element of the field extension $\mathbb{F}_{q^{N}}/\mathbb{F}_{q}$ and denote by $f(x)$ the irreducible polynomial of $u$ over $\mathbb{F}_{q}$. Since we have an $\overline{\bbf}_{q}$-algebra isomorphism 
\[
\overline{\bbf}_{q}\otimes_{\mathbb{F}_{q}}\mathbb{F}_{q^{N}}\cong\overline{\bbf}_{q}\otimes_{\mathbb{F}_{q}}(\mathbb{F}_{q}[x]/(f(x)))\cong \overline{\bbf}_{q}[x]/(x-u)\times...\times\overline{\bbf}_{q}[x]/(x-u^{q^{N-1}})
\]
by the Chinese Remainder Theorem, we deduce that
\[
\begin{aligned}
 {\rm  Hilb}^{n}({\rm Spec}(\widehat{\mathcal{O}_{S,y}})/{\rm Spec}\mathbb{F}_{q})\otimes_{\mathbb{F}_{q}}\overline{\bbf}_{q}&\cong{\rm Hilb}^{n}({\rm Spec}((\overline{\bbf}_{q}\times...\times\overline{\bbf}_{q})[[s,t]])/{\rm Spec}\overline{\bbf}_{q})\\
&\cong{\rm Hilb}^{n}(\coprod {\rm Spec}\overline{\bbf}_{q}[[s,t]]/{\rm Spec}\overline{\bbf}_{q}).
\end{aligned}
\]
Hence the $\overline{\bbf}_{q}$-valued points of ${\rm  Hilb}^{n}({\rm Spec}(\widehat{\mathcal{O}_{S,y}})/{\rm Spec}\mathbb{F}_{q})$ correspond to the closed subschemes of degree $n$ of $\coprod {\rm Spec}\overline{\bbf}_{q}[[s,t]]$, i.e. the closed subschemes of degree $n$ of $S$ whose underlying space is a subset of the points $x, F_{q}(x),..., F_{q}^{N-1}(x)$.

Since $F_{q}$ acts on $\mathbb{F}_{q^{N}}[[s,t]]$ by sending $s$ to $s^{q}$, $t$ to $t^{q}$ and $c\in\mathbb{F}_{q^{N}}$ to $c^{q}$, we deduce from the above discussion that $F_{q}$ acts on $(\overline{\bbf}_{q}\times...\times\overline{\bbf}_{q})[[s,t]]$ by sending $s$ to $s^{q}$, $t$ to $t^{q}$ and $(\alpha_0,\alpha_1,...,\alpha_{N-2},\alpha_{N-1})\in \overline{\bbf}_{q}\times...\times\overline{\bbf}_{q}$ to $(\alpha_1,\alpha_{2},...,\alpha_{N-1},\alpha_{0})$. This is actually an algebraic assertion, which can also be seen geometrically. For example, $F_{q}$ is a $\overline{\bbf}_{q}$-morphism from $\widehat{\mathcal{O}_{S_{\overline{\bbf}_{q}},F_{q}(x)}}\cong(\{0\}\times\overline{\bbf}_{q}\times...\times\{0\})[[s,t]]$ to $\widehat{\mathcal{O}_{S_{\overline{\bbf}_{q}},x}}\cong(\overline{\bbf}_{q}\times\{0\}\times...\times\{0\})[[s,t]]$.

Let $\sigma$ be an element of ${\rm Gal}(\mathbb{F}_{q^{N}}/\mathbb{F}_{q})$. Recall that for an $\mathbb{F}_{q^{N}}$-vector space $V$, a $\sigma$-linear map $f:V\rightarrow V$ is an additive map on $V$ such that $f(\alpha v)=\sigma(\alpha)f(v)$ for all $\alpha\in\mathbb{F}_{q^{N}}$ and $v\in V$.

\begin{lemma}
Let $H=\left<g\right>$. Suppose $p\nmid|H|$. Then we can choose $s$ and $t$ such that $g$ acts on $\mathbb{F}_{q^{N}}[[s,t]]$ $\sigma$-linearly, where $\sigma$ is the inverse of the Frobenius automorphism of ${\rm Gal}(\mathbb{F}_{q^{N}}/\mathbb{F}_{q})$.
\end{lemma}

\begin{proof}
The automorphism $g$ acts as an $\mathbb{F}_{q}$-automorphism on $\mathbb{F}_{q^{N}}[[s,t]]$ fixing the ideal $(s,t)$ and sending $\mathbb{F}_{q^{N}}$ to $\mathbb{F}_{q^{N}}$. Since we know $F_{q}$ sends $(\alpha_0,\alpha_1,...,\alpha_{N-2},\alpha_{N-1})\in \overline{\bbf}_{q}\times...\times\overline{\bbf}_{q}$ to $(\alpha_1,\alpha_{2},...,\alpha_{N-1},\alpha_{0})$ and $gF_{q}$ fixes the geometric points $x, F_{q}(x),..., F_{q}^{N-1}(x)$, we deduce that $g$ sends $(\alpha_0,\alpha_1,...,\alpha_{N-2},\alpha_{N-1})\in \overline{\bbf}_{q}\times...\times\overline{\bbf}_{q}$ to $(\alpha_{N-1},\alpha_{0},...,\alpha_{N-3},\alpha_{N-2})$. Hence $g(\alpha)=\sigma(\alpha)$ for all $\alpha\in\mathbb{F}_{q^{N}}$ where $\sigma$ is the inverse of the Frobenius automorphism.

For any element $h\in H$, we write $h(s)=as+bt+...$ and $h(t)=cs+dt+...$ where $a,b,c,d\in\mathbb{F}_{q}$ since $h$ commutes with $F_{q}$. Define an automorphism $\rho(h)$ of $\mathbb{F}_{q^{N}}[[s,t]]$ by $\rho(h)(s)=as+bt$, $\rho(h)(t)=cs+dt$ and the action of $\rho(h)$ on $\mathbb{F}_{q^{N}}$ is the same as the action of $h$. Then we denote the $\mathbb{F}_{q^{N}}$-automorphism $\frac{1}{|H|}\sum_{h\in H}h\rho(h)^{-1}$ by $\theta$. Notice that $\theta$ is an automorphism because the linear term of $\theta$ is an invertible matrix, and here is the only place we use the assumption that $p\nmid|G|$. We deduce that $g\theta=\theta\rho(g)$, which implies $\theta^{-1}g\theta=\rho(g)$. Hence we are done. 
\end{proof}

The above discussion implies that the $g$-action on $(\overline{\bbf}_{q}\times...\times\overline{\bbf}_{q})[[s,t]]$ is given by sending $s$ to $(a,...,a)s+(b,...,b)t$, $t$ to $(c,...,c)s+(d,...,d)t$ and $(\alpha_0,\alpha_1,...,\alpha_{N-2},\alpha_{N-1})\in \overline{\bbf}_{q}\times...\times\overline{\bbf}_{q}$ to $(\alpha_{N-1},\alpha_{0},...,\alpha_{N-3},\alpha_{N-2})$.

Hence the action of $gF_{q}$ on $(\overline{\bbf}_{q}\times...\times\overline{\bbf}_{q})[[s,t]]$ is given by sending $s$ to $(a,...,a)s^{q}+(b,...,b)t^{q}$, $t$ to $(c,...,c)s^{q}+(d,...,d)t^{q}$ and $(\alpha_0,\alpha_1,...,\alpha_{N-2},\alpha_{N-1})\in \overline{\bbf}_{q}\times...\times\overline{\bbf}_{q}$ to itself. This implies that $gF_{q}$ acts on each complete local ring, which is what we expected since $gF_{q}$ fixes each geometric point over $y$. In particular, it acts on $\widehat{\mathcal{O}_{S_{\overline{\bbf}_{q}},x}}\cong(\overline{\bbf}_{q}\times\{0\}\times...\times\{0\})[[s,t]]\cong\overline{\bbf}_{q}[[s,t]]$.

Recall that ${\rm Hilb}^{n}(\widehat{\mathcal{O}_{S_{\overline{\bbf}_{q}},x}})(\overline{\bbf}_{q})$ parametrizes closed subschemes of degree $n$ of $S_{\overline{\bbf}_{q}}$ supported on $x$.

\begin{proof}[Proof of Lemma \ref{fixpoint}]
First we define an $\mathbb{F}_{q}$-automorphism $\tilde{g}$ on $\mathbb{F}_{q}[[s,t]]$ by 
\[
\tilde{g}(s)=as+bt\ \ {\rm and}\ \ \tilde{g}(t)=cs+dt
\]
Recall that the action of $F_{q}$ on $\mathbb{F}_{q}[[s,t]]$ is an $\mathbb{F}_{q}$-endomorphism sending $s$ to $s^{q}$ and $t$ to $t^{q}$. By the above discussion, we observe that the action of $gF_{q}$ on $\overline{\bbf}_{q}[[s,t]]$ on the left is the same as the action of $\tilde{g}{F_{q}}$ on $\overline{\bbf}_{q}[[s,t]]$ on the right. Hence we have
\[
|{\rm Hilb}^{n}(\widehat{\mathcal{O}_{S_{\overline{\bbf}_{q}},x}})(\overline{\bbf}_{q})^{gF_{q}}|=|{\rm Hilb}^{n}(\mathbb{F}_{q}[[s,t]])(\overline{\bbf}_{q})^{\tilde{g}{F_{q}}}|.
\]
Now for the right hand side, $\tilde{g}$ is an automorphism of finite order and $F_{q}$ is the geometric Frobenius. Then by the Grothendieck trace formula, we have
\[
|{\rm Hilb}^{n}(\mathbb{F}_{q}[[s,t]])(\overline{\bbf}_{q})^{\tilde{g}{F_{q}}}|=\sum_{k=0}^{\infty}(-1)^{k}{\rm Tr}((\tilde{g}{F_{q}})^{*}, H^{k}_{c}({\rm Hilb}^{n}(\overline{\bbf}_{q}[[s,t]]),\mathbb{Q}_{l})).
\] 
But the action of $\tilde{g}$ factors through $GL_{2}(\overline{\bbf}_{q})$. Now we use the fact that if $G$ is a connected algebraic group acting on a separated and finite type scheme $X$, then the action of $g\in G$ on $H^{*}_{c}(X,\mathbb{Q}_{l})$ is trivial \cite[Corollary 6.5]{DL76}. Hence we have
 \[
\begin{aligned}
|{\rm Hilb}^{n}(\mathbb{F}_{q}[[s,t]])(\overline{\bbf}_{q})^{\tilde{g}{F_{q}}}|&=\sum_{k=0}^{\infty}(-1)^{k}{\rm Tr}(({F_{q}})^{*}, H^{k}_{c}({\rm Hilb}^{n}(\overline{\bbf}_{q}[[s,t]]),\mathbb{Q}_{l}))\\
&=|{\rm Hilb}^{n}(\mathbb{F}_{q}[[s,t]])(\overline{\bbf}_{q})^{F_{q}}|.
\end{aligned}
\] 
\end{proof}

Suppose $S$ is a smooth projective K3 surface over $\overline{\bbf}_{p}$ with a G-action. Recall that a \emph{Mathieu representation} of a finite group G is a 24-dimensional representation on a vector space $V$ over a field of characteristic zero with character
\[
\chi(g)=\epsilon(\ord(g)),
\]
where
\[
\epsilon(n)=24(n\prod_{p|n}(1+\frac{1}{p}))^{-1}.
\]

\begin{prop}\cite[Proposition 4.1]{DK09}\label{symplectic fact}
Let G be a finite group of symplectic automorphisms of a K3 surface $X$ defined in characteristic $p>0$. Assume that $p\nmid G$. Then for any prime $l\neq p$, the natural representation of G on the $l$-adic cohomology groups $H^{*}(X,\bbq_{l})\cong\bbq_{l}^{24}$ is Mathieu.
\end{prop}

\begin{proof}[Proof of theorem \ref{thm2}]
By Remark \ref{rmk}, we deduce that 
\[
\sum_{n=0}^{\infty}\text{Tr}(g,[e({S}^{[n]})])t^{n}=\text{exp}\left(\sum_{m=1}^{\infty}\sum_{k=1}^{\infty}\frac{\text{Tr}(g^{k},[e({S})])t^{mk}}{k}\right).
\]
Then by Proposition \ref{symplectic fact}, we obtain the equality we want. 

When $G$ is a cyclic group of order $N$, we know that $N\leq 8$ by \cite[Theorem 3.3]{DK09}. Then the proof is the same as the proof in \cite{Zha21} in the characteristic zero case.
\end{proof}

\begin{proof}[Proof of theorem \ref{thm3}]
If $g$ is a symplectic automorphism (fixing the origin) on a complex abelian surface, then $g$ has order $1,2,3,4$ or $6$ by \cite[Lemma 3.3]{Fuj88}. We will do the case when the order $N=4$, and the calculation for other cases are similar. By \cite[Page 33]{Fuj88}, we know the explicit action of $g$ on the torus $S=\bbc^2/\bigwedge$ in each case. If $N=4$, then the action on $H^1(S,\bbc)$ is given by
\[
\begin{pmatrix}
0 & -1 & 0 & 0\\
1 & 0  & 0 & 0\\
0 & 0 & 0 & -1\\
0 & 0 & 1 & 0
\end{pmatrix}.
\]  
Hence we deduce that $\Tr(g, [e(S)])=\Tr(g^3,[e(S)])=4$, and $\Tr(g^2,[e(S)])=16$. Now
\[
\begin{aligned}
\sum_{n=0}^{\infty}\text{Tr}(g,[e({S}^{[n]})])t^{n}&=\exp\left(\sum_{m=1}^{\infty}\sum_{k=1}^{\infty}\frac{\text{Tr}(g^{k},[e({S})])t^{mk}}{k}\right)\\
&=\exp\left(\sum_{m=1}^{\infty}\sum_{k\equiv 1,3}\frac{4t^{mk}}{k}+\sum_{m=1}^{\infty}\sum_{k\equiv 2}\frac{16t^{mk}}{k}\right)\\
&=\exp\left(\sum_{m=1}^{\infty}\left(\sum_{k=1}^{\infty}\frac{4t^{mk}}{k}-\sum_{k=1}^{\infty}\frac{4t^{2mk}}{2k}\right)+\sum_{m=1}^{\infty}\left(\sum_{k=1}^{\infty}\frac{16t^{2mk}}{2k}-\sum_{k=1}^{\infty}\frac{16t^{4mk}}{4k}\right)\right)\\
&=\frac{\prod_{m=1}^{\infty}(1-t^m)^{-4}}{\prod_{m=1}^{\infty}(1-t^{2m})^{-2}}\frac{\prod_{m=1}^{\infty}(1-t^{2m})^{-8}}{\prod_{m=1}^{\infty}(1-t^{4m})^{-4}}\\
&=\frac{\eta^{4}(t^4)}{\eta^{4}(t)\eta^{6}(t^2)}.
\end{aligned}
\]
\end{proof}

\begin{remark}
Fix a smooth projective surface $S$ and an automorphism $g$ of finite order. From the proof of Theorem \ref{thm2} or Theorem \ref{thm3}, we notice that if  $\Tr(g^{k},[e({S})])$ only depends on the order of $g^k$ in the cyclic group $\langle g\rangle$ for $k\geq 0$, then the generating function $\sum_{n=0}^{\infty}\text{Tr}(g,[e({S}^{[n]})])t^{n}$ is an eta quotient by the inclusion-exclusion principle.
\end{remark}

\section{Generalized Kummer varieties}

Let $A$ be an abelian surface over $\bbc$. Let $\omega_n:A^{[n]}\to A^{(n)}$ be the Hilbert-Chow morphism and let $g_n: A^{(n)}\to A$ be the addition map. The generalized Kummer variety of $A$ is defined to be 
\[
K_{n}(A):=\omega_{n}^{-1}(g_{n}^{-1}(0)).
\]
This is a smooth projective holomorphic symplectic variety. We follow the strategy in \cite{Go94}.

Now suppose $A$ is an abelian surface with a $G$-action over $\bbf_q$. Define the map $\gamma_n$ by
\[
\gamma_n: \coprod_{\alpha\in P(n)}\left(\left(\prod_{i=1}^{\infty}S^{(\alpha_i)}(\overline{\bbf}_q)^{gF_q}\right)\times\bba^{n-|\alpha|}({\bbf}_q) \right)\to S^{(n)}(\overline{\bbf}_q)^{gF_q},
\]
\[
((\zeta_i)_i,v)\mapsto\sum i\cdot\zeta_i.
\]

\begin{lemma}\label{preimage}
For any $\zeta\in S^{(n)}(\overline{\bbf}_q)^{gF_q}$, we have $|\gamma_{n}^{-1}(\zeta)=|\omega_{n}^{-1}(\zeta)|$.
\end{lemma}
\begin{proof}
Let $\zeta=\sum_{i=1}^{r}n_{i}\zeta_i\in S^{(n)}(\overline{\bbf}_q)^{gF_q}$, where $\zeta_i$ are distinct primitive cycles of degree $d_i$. Then
\[
\begin{aligned}
|\omega_{n}^{-1}(\zeta)|&=\prod_{i=1}^{r}|V_{n_i}(\overline{\bbf}_q)^{(gF_q)^{d_i}}|\\
&=\prod_{i=1}^{r}|V_{n_i}(\overline{\bbf}_q)^{(F_q)^{d_i}}|\\
&=\prod_{i=1}^{r}\sum_{\beta^{i}_j\in P(n_i)}q^{d_i(n_i-|\beta^{i}_j|)},
\end{aligned}
\]
where $V_n={\rm Hilb}^{n}(\mathbb{F}_{q}[[s,t]])$. Here we use the key Lemma \ref{fixpoint}.

For $i=1,...,r$, let $\beta^{i}=(1^{\beta^{i}_1},2^{\beta^{i}_2},...)$ be a partition of $n_i$, and let $\alpha=(1^{\alpha_1},2^{\alpha_2},...)$ be the union of $d_i$ copies of each $\beta^i$, where $\alpha_j=\sum_{i}d_{i}\beta^{i}_{j}$. Let
\[
\eta_j=\sum_{i=1}^{r}\beta^{i}_{j}\zeta_i.
\]
Let $\eta$ be the sequence $(\eta_1,\eta_2,\eta_3,...)$. Then for all $w\in \bba^{n-|\alpha|}$ we have
\[
\gamma_{n}((\eta,w))=\zeta,
\]
and in this way we get all the elements of $\gamma_{n}^{-1}(\zeta)$. Hence
\[
|\gamma_{n}^{-1}(\zeta)|=\sum_{\beta^{1}\in P(n_1)}\sum_{\beta^{2}\in P(n_2)}\cdots\sum_{\beta^{r}\in P(n_r)}q^{n-\sum d_{i}|\beta^i|}=|\omega_{n}^{-1}(\zeta)|.
\]
\end{proof}

\begin{lemma}\label{lemma1}
Denote by $h_n: A^{(n)}(\overline{\bbf}_q)^{gF_q}\to A(\overline{\bbf}_q)^{gF_q}$ the restriction of $g_n$. Then $h_n$ is onto and $|h_{n}^{-1}(x)|$ is independent of $x\in A(\overline{\bbf}_q)^{gF_q}$.
\end{lemma}
\begin{proof}
Since $gF_q$ is the Frobenius map of some twist of $A$, we can replace $gF_q$ by $F_q$ in the statement, and this is true by \cite[Lemma 2.4.8]{Go94}.
\end{proof}

For each $l\in\mathbb{N}$, let $A(\overline{\bbf}_q)^{gF_q}_{l}$ be the image of the multiplication $(l):A(\overline{\bbf}_q)^{gF_q}\to A(\overline{\bbf}_q)^{gF_q}$.

\begin{lemma}\label{lemma2}
Let $\mu=(n_1,...,n_t)$ be a partition of a number $n\in\mathbb{N}$. Then
\[
\sigma_{\mu}\colon (A(\overline{\bbf}_q)^{gF_q})^{t}\to A(\overline{\bbf}_q)^{gF_q}_{gcd(\mu)}
\]
\[
(x_1,...,x_t)\mapsto\sum_{i=1}^{t}n_{i}x_{i}
\]
is onto and $|\sigma^{-1}_{\mu}(x)|$ is independent of $x\in A(\overline{\bbf}_q)^{gF_q}_{gcd(\mu)}$.
\end{lemma}
\begin{proof}
As the above lemma, we can replace $gF_q$ by $F_q$, and this is true by \cite[Lemma 2.4.9]{Go94}.
\end{proof}

We denote $(\left(\prod_{i=1}^{\infty}S^{(\alpha_i)}(\overline{\bbf}_q)^{gF_q}\right)\times\bba^{n-|\alpha|}({\bbf}_q)$ by $A[\alpha]$. Denote the restriction map of $\gamma_n$ on $A[\alpha]$ by $\gamma_{n,\alpha}:A[\alpha]\to S^{(n)}(\overline{\bbf}_q)^{gF_q}$.

\begin{lemma}\label{kummer}
\[
|K_{n}(A)(\overline{\bbf}_q)^{gF_q}|=\frac{1}{|A(\overline{\bbf}_q)^{gF_q}|}\sum_{\alpha\in P(n)}\left(gcd(\alpha)^{4}q^{n-|\alpha|}\prod_{i=1}^{\infty}|A^{(\alpha_i)}(\overline{\bbf}_q)^{gF_q}|\right).
\]

\end{lemma}
\begin{proof}
By Lemma \ref{preimage}, we have
\[
|K_{n}(A)(\overline{\bbf}_q)^{gF_q}|=|\gamma_{n}^{-1}(h_{n}^{-1})|=\sum_{\alpha\in P(n)}|\gamma_{n,\alpha}^{-1}(h_{n}^{-1}(0))|.
\]
Suppose $\alpha=(1^{\alpha_1},2^{\alpha_2},...)$. Let
\[
\mu=(m_1,...,m_t)\colon =(1^{\mu_1},2^{\mu_2},...),
\]
where $\mu_i={\rm min}(1,\alpha_i)$ for all $i$. Let
\[
f_{\alpha}: S[\alpha]\to (A(\overline{\bbf}_q)^{gF_q})^{t}
\]
\[
((\zeta_1,...,\zeta_t),w)\mapsto (g_{\alpha_{m_1}}(\zeta_1),...,g_{\alpha_{m_t}}(\zeta_t)).
\]
Then the following diagram commutes:
\[
\begin{CD}
A[\alpha] @>\gamma_{n,\alpha}>> A^{(n)}(\overline{\bbf}_q)^{gF_q}\\
@V f_{\alpha} VV @VV h_n V\\
(A(\overline{\bbf}_q)^{gF_q})^t @>\sigma_{\mu}>> A(\overline{\bbf}_q)^{gF_q}.
\end{CD}
\]
By Lemma \ref{lemma1} and Lemma \ref{lemma2}, $\sigma_{\mu}\circ f_{\alpha}$ maps $S[\alpha]$ onto $A(\overline{\bbf}_q)^{gF_q}_{gcd(\alpha)}=A(\overline{\bbf}_q)^{gF_q}_{gcd(\mu)}$, and $|f_{\alpha}^{-1}(\sigma_{\mu}^{-1}(x))|$ is independent of $x\in A(\overline{\bbf}_q)^{gF_q}_{gcd(\alpha)}$. Since the multiplication with $gcd(\alpha)$ is an \'etale morphism of degree $(gcd(\alpha))^4$, we have
\[
\begin{aligned}
|K_{n}(A)(\overline{\bbf}_q)^{gF_q}|&=\sum_{\alpha\in P(n)}|f_{\alpha}^{-1}(\sigma_{\mu}^{-1}(x))|\\
&=\sum_{\alpha\in P(n)}\frac{|A[\alpha]|}{|A(\overline{\bbf}_q)^{gF_q}_{gcd(\alpha)}|}\\
&=\frac{1}{|A(\overline{\bbf}_q)^{gF_q}|}\sum_{\alpha\in P(n)}\left(gcd(\alpha)^{4}q^{n-|\alpha|}\prod_{i=1}^{\infty}|A^{(\alpha_i)}(\overline{\bbf}_q)^{gF_q}|\right).
\end{aligned}
\]
\end{proof}

\begin{proof}[Proof of theorem \ref{thm4}]
By Lemma \ref{kummer}, we have
\[
\begin{aligned}
\sum_{n=0}^{\infty}|K_{n}(A)(\overline{\bbf}_q)^{gF_q}|t^{n}&=\sum_{n=0}^{\infty}\frac{1}{|A(\overline{\bbf}_q)^{gF_q}|}\sum_{\alpha\in P(n)}\left(gcd(\alpha)^{4}q^{n-|\alpha|}\prod_{i=1}^{\infty}|A^{(\alpha_i)}(\overline{\bbf}_q)^{gF_q}|\right)t^n\\
&=\frac{(w\frac{d}{dw})^4}{|A(\overline{\bbf}_q)^{gF_q}|}\sum_{n=0}^{\infty}\sum_{\alpha\in P(n)}w^{gcd(\alpha)}\prod_{i=1}^{\infty}\Big(|A^{(\alpha_i)}(\overline{\bbf}_q)^{gF_q}| q^{(i-1)\alpha_{i}}t^{i\alpha_i}\Big)\\
&=\frac{(w\frac{d}{dw})^4}{|A(\overline{\bbf}_q)^{gF_q}|}\bigodot_{m=1}^{\infty}\left(1+w^{m}(-1+\sum_{n=0}^{\infty}|A^{(n)}(\overline{\bbf}_q)^{gF_q}|q^{(m-1)n}t^{mn})\right).
\end{aligned}
\]
Then by the proof of Proposition \ref{symmetric} and Theorem \ref{padichodge}, the theorem follows.
\end{proof}

\begin{remark}
It is calculated in \cite[Corollary 2.4.13]{Go94} that $\sum_{n=1}^{\infty}e(K_{n}(A))q^{n}=\frac{(q\frac{d}{dq})^{3}}{24}E_2$, where $E_2:=1-24\sum_{n=1}^{\infty}\sigma_{1}(n)q^{n}$ is a quasi-modular form. As in the case of Hilbert schemes of points, we can calculate $\sum_{n=0}^{\infty}{\Tr(g,[e(K_{n}(A))])t^n}$, where $g$ is a symplectic automorphism of finite order on the abelian surface $A$. But it is not obvious to the author whether or not the sum can be expressed by quasi-modular forms.
\end{remark}

\end{document}